\documentclass[11pt]{amsart}

\usepackage{enumerate,url,amssymb,  mathrsfs, graphicx, pdfsync}
\newtheorem{theorem}{Theorem}

\newtheorem{lemma}[theorem]{Lemma}
\newtheorem*{lemma*}{Lemma}

\newtheorem{corollary}[theorem]{Corollary}
\theoremstyle{definition}
\newtheorem{definition}[theorem]{Definition}
\newtheorem{example}[theorem]{Example}

\theoremstyle{remark}
\newtheorem{remark}[theorem]{Remark}

\numberwithin{equation}{section}

\newcommand{\R}{\mathbb{R}}

\newcommand{\onto}{\xrightarrow[]{{}_{\!\!\textnormal{onto\,\,}\!\!}}}
\newcommand{\into}{\xrightarrow[]{{}_{\!\!\textnormal{into\,\,}\!\!}}}

\newcommand{\bydef}{\stackrel{\textnormal{def}}{=\!\!=}}

\begin{document}

\pagestyle{plain}

\title{Complex  Harmonic Capacitors}

\author[T. Iwaniec]{Tadeusz Iwaniec}
\address{Department of Mathematics, Syracuse University, Syracuse,
NY 13244, USA}
\email{tiwaniec@syr.edu}

\author[J. Onninen]{Jani Onninen}
\address{Department of Mathematics, Syracuse University, Syracuse,
NY 13244, USA and  Department of Mathematics and Statistics, P.O.Box 35 (MaD) FI-40014 University of Jyv\"askyl\"a, Finland
}
\email{jkonnine@syr.edu}

\author[T. Radice]{Teresa Radice}
\address{Universit\'a degli Studi di Napoli ``Federico II'', Dipartimento di Matematica e Applicazioni ``R. Caccioppoli''
Via Cintia, 80126 Napoli, Italy}
\email{teresa.radice@unina.it}
	\date{\today}

\thanks{J. Onninen was supported by the NSF grant  DMS-2154943.}

\maketitle
\begin{abstract}
The concept of complex harmonic potential in a doubly connected condenser (capacitor)  is introduced as an analogue of the real-valued potential of an electrostatic vector field.  In this analogy the full differential
 of a complex potential plays the role of  the gradient of the scalar potential in the theory of electrostatic.
 The main objective in the non-static fields is to rule out having the  full differential vanish at some points. Nevertheless, there can be critical points  where the Jacobian determinant of the differential turns into zero. The latter is in marked contrast to the case of real-valued potentials.  Furthermore, the complex electric capacitor also admits an interpretation of the stored energy  intensively studied in the theory of hyperelastic deformations. 
 
Engineers interested in electrical systems, such as 
energy storage devises,  might also wish to envision complex capacitors as {\bf electromagnetic condensers} which,  generally, store more energy that the electric capacitors.
\end{abstract}

\section{Introduction} Let us first set up the key terminology.
\subsection{About domains}
 Throughout this text $\,\Omega \subset \mathbb R^2\,$ will be a bounded simply connected domain.  Thus its complement in the Riemann sphere  $\,\widehat{\mathbb R^2} = \mathbb S^2\,$, denoted by $\, \mathbf{G }\bydef \widehat{\mathbb R^2} \setminus \Omega\,$, is a continuum    (compact and connected).  Similarly, a bounded domain  $\,\Delta \subset \mathbb R^2\,$ is doubly connected if its complement $\,\widehat{\mathbb R^2} \setminus \Delta\,$ consists of two disjoint continua. These  definitions agree with a widely used homological characterization of multiple-connectivity. Since $\,\Delta\,$ is a domain (open and connected),  none of its two complementary continua disconects $\,\widehat{\mathbb R^2} \,$. 
 
 There  is another classical characterization of a doubly connected domain. Just remove from a simply connected domain $\,\Omega\,$ a nonempty continuum $\,\digamma \subset \Omega\,$ that does not disconnect $\,\mathbb C\,$; that is, $\,\mathbb C \setminus \digamma\,$ is connected. The boundary of  $\,\Delta= \Omega \setminus \digamma \,$ consists of two disjoint continua: one is just  $\,\partial \digamma\,$, called the \textit{inner boundary} of $\,\Delta\,$ ,  and the other one is  $\, \partial \mathbf {G} = \partial \Omega\,$, called the \textit{outer boundary} of $\,\Delta\,$.
In what follows $\,\Delta\,$ will be interpreted as  dielectric medium between two electrical conductors; that is, between $\,\partial \digamma\,$ and $\partial \Omega\,$.

\subsection{Electrostatic field}
Let us begin with the real-valued  function as a potential of an electrostatic field. The commonly used term potential of a vector field  $\,\mathbf E : \Delta \rightarrow \mathbb R^2\,$ in a domain $\,\Delta \subset \mathbb R^2\,$ refers to  a scalar function $\, u : \Delta \rightarrow  \mathbb R\,$  whose gradient defines the vector field; in symbols,  $\,\textbf{E} = \nabla u\,$. Quite often $\,u\,$ is treated as the potential of an electrostatic gradient field which minimizes an  energy integral,  subject to given boundary values. It is important to emphasize that \textit{non-static fields} cannot be represented through a scalar potential in this manner. This stands as the principal concern addressed in the present paper.

Before all else,  the reader may wish to envision the boundary  of  a simply connected domain $\,\Omega \subset \mathbb R^2\,$ as negatively charged conductor.  Then, an imaginary moving particle in $\Omega$ that is positively charged is pulled towards $\,\partial \Omega\,$ by the electric force.  As it approaches $\partial \Omega$, the electric field intensifies progressively.

\subsection{Capacitor / or condenser}\label{CondenserCapacitor}

Now suppose we have a nonempty continuum $\, \digamma \Subset \Omega\,$ (compact connected subset)  that does not disconnect $\,\mathbb C\,$. This gives rise to a doubly connected region  $\,\Delta  \bydef \Omega \setminus  \digamma\,$ for an electric field between $\,\partial \Omega\,$ and $\,\digamma\,$. Hereafter $\,\digamma\,$ is positively charged.  Such a pair $\, (\Omega , \digamma )\,$ is interpreted  as a capacitor (or condenser)  in which  an electric potential difference (a voltage) is applied across the terminals $\, \partial \Omega\,$ and $\partial \digamma$. Let us look at the capacitors  with two electrical conductors separated by a \textit{dielectric medium}  $\,\Delta = \Omega \setminus\digamma\,$, like $\,\partial \Omega\,$ and $\partial  \digamma$. Later we shall assume that  $\,\digamma\,$ does not degenerate to a point (because the isolated points are removable singularities for bounded harmonic functions).  In this way the complex potential  is a function $\,H : \Delta \rightarrow \mathbb C\,$. It defines a force that works to make the imaginary particles move from  $\, \digamma\,$ toward $\,\partial \Omega\,$. This force will be represented by the \textit{complex  differential 1-form} $\, \textnormal d H  \bydef  H_z \,\textnormal d  z  \; +  \; H_{\overline{z} }\,\textnormal d \overline{z}\,$.

\subsection{Nowhere vanishing static gradient fields}
Let us look again at the case  $\,\textbf{E } = \nabla  u\,$, where the scalar harmonic function $\,u\,$ assumes two different constant values, one  on each of two conductors (the boundary components of $\,\Delta\,$).

It is well known that a static \textit{harmonic field}  $\,\textbf{E } = \nabla  u\,$, with different constant values  of the potential function $\,u\,$ along $\,\partial \Omega\,$ and $\,\digamma\,$; say $\, u \equiv 0 \,$ on $\,  \partial \Omega\,$ and $\, u \equiv 1 \,$ on $\, \digamma \,$, does not vanish in $\,\Delta = \Omega \setminus \digamma\,$.  The proof can be easily  brought down to the case of a round annulus. Indeed, with the aid of a conformal change of variables, say   $\, \phi :  A(r,R) \bydef \{\xi :\;  r < |\xi| < R\,\} \onto \,\Delta\,$ (continuous up to the boundary or not),   we are led to examine the harmonic function $\, v(\xi) = u(\phi(\xi))\,$. This function, being uniquely determined by its boundary values,  turns out to be equal to   $ \;\log \frac {|\xi|}{R}\,/ \log \frac{r}{R}  \,$. Hence $ \nabla u (z) \not = 0$ for all $z\in \Delta$.  
Unfortunately, such a straightforward conformal change of variables does not always work. 
More comprehensive idea of the proof (for harmonic fields) is illustrated in the forthcoming Corollary \ref{scalarcase}, including the use of harmonic dendrites. These ideas also extend to $p$-harmonic fields.
Nonetheless, the following  nonlinear $\,p\,$-harmonic variant of this fact merits mentioning. 
\begin{theorem} \label{pharmonicGradient}
The gradient field $\,\textbf{E } = \nabla  u\,$, of a non-constant $\,p\,$-harmonic function $\,u : \Delta \rightarrow \mathbb R\,$,  that is  continuous on $\,\overline{\Delta}\,$ and assumes constant values on each of its  boundary components $\,\partial \digamma\,$ and $\,\partial \Omega\,$, does not vanish in  $\Delta = \Omega\setminus \digamma\,$.  
\end{theorem}
However, {Completely different approach is needed to deal with complex-valued potential functions.}

\subsection{Nowhere vanishing non-static full differential fields}
We finish the introduction by stating a particular case and  a direct consequence of our main result which is Theorem \ref{MainTheorem} in Section~\ref{sec:mainresult}.

\begin{theorem}\label{SpecialCase}
Let a continuous function  $\, H : \partial \Delta   \; \rightarrow \mathbb C\,$ satisfy:
\begin{itemize}
\item $\, H\,$ is constant on $\,\partial \digamma\,$
\item $\,H\,: \partial \Omega \onto \partial\mathcal Q\,$ is a homeomorphism (or a monotone  map) onto  the boundary of a convex domain $\,\mathcal Q\,$.
\end{itemize} 
Then the differential of the harmonic extension of $\,H\,$, still denoted by $\,H : \overline{\Delta}  \rightarrow \mathbb C\,$, does not vanish in $\,\Delta\,$. In symbols,
\begin{equation} \label{NonvanishingDifferential}
 \;\;\;\frac{1}{2}\, \big{\Vert }\;\textnormal d H(z) \;\big{\Vert }^2 \,=\,  \vert H_z(z)\vert^2    \; +  \; \vert H_{\overline{z} }(z)\vert^2\; \neq 0\;\;,\;\textnormal{ in}\;\; \Delta\,.
\end{equation}
\end{theorem}

However, the Jacobian determinant may vanish at some points, see the example in  Section  \ref{Example}.
\begin{equation} \label{VanishingJacobian}
 \mathcal J_H(z)  \,\bydef \,  \vert H_z(z)\vert^2    \; -  \; \vert H_{\overline{z} }(z)\vert^2\; = 0 \;,\;\textnormal{ for some}\;\; z \in \Delta .
\end{equation}
 That being the case, the mapping  $\,H : \Delta \rightarrow \mathbb C\,$ may fail to be injective, contrasting with the classical  Rad\'{o}-Kneser-Choquet theorem (RKC) (Theorem \ref{pHarmonicRKC}, below) for simply connected domains.
\begin{remark} \label{NonvanishingDifferential} It seems that within  doubly connected variants of RKC-theorem, Theorem~\ref{SpecialCase}, there is a general rule : 
$$ \textnormal{\textit{The differential} }\; 
\,\textnormal d H  \bydef  H_z \,\textnormal d  z  \; +  \; H_{\overline{z} }\,\textnormal d \overline{z}\; \textnormal{ \textit{cannot vanish anywhere ? }} 
$$
\end{remark}

In this context it is tempting to turn to the celebrated J.C.C Nitsche conjecture, see \cite{Nitsche, Lyzzaik,  NitscheConjecture}

Apropos, we would like also to draw the readers attention to the 3D-case.  Laugesen
constructed in \cite{Laugesen} a self-homeomorphism of the sphere whose
harmonic extension to the ball in $\R^3$ fails to be injective.  However,  this construction is impossible for homeomorphisms which are gradients  of a scalar harmonic function. Such a harmonic map is automatically a diffeomorphism~\cite{GleasonWolff, Lewy}. We, however, do not enter to the higher dimensional case here.

\section{Complex Capacitors}  In analogous setting the term complex capacitor (or condenser)  refers to a device that stores energy of the differential of a complex-valued potential function, which is typically specified on the boundary components. Another way to interpret complex capacitors is via mathematical models of hyperelasticity \cite{Antman, Ball1, Ball2, Ball3, Ball4, Ciarlet}. In these models the  harmonic mapping $\, H\,: \Delta \rightarrow \mathbb C\,$ is understood  as a hyperelastic deformation of a plate.  Subsequently, the so-called star-shape of the deformed boundary $\, H(\partial\Delta)\,$ will be imposed and natural. Regrettably, even under this asssumption it is not generally possible to ensure injectivity of $\, H\,: \Delta \rightarrow \mathbb C\,$. We say,  an \textit{interpenetration of matter}  can occur. 

\subsection{Elecromagnetic capacitor} The complex potential  $\,H =  U +  i V\,$  might be visualized as that of an \textit{electromagnetic capacitor} which combines \textit{electrostatic charge}  $\, E \bydef \nabla U\,$ with  a magnetostatic field  $\, B \bydef \nabla V\,$.  Gauss's low reads as:
$$
\textnormal{div} E \,= \Delta U \equiv 0\;\;\; \textnormal{and}\;\; \textnormal{div} B \,= \Delta V \equiv 0\,.\
$$
The complex differential  $\, \textnormal {d} H \bydef\, H_z \,\textnormal{d} z \; +\;   H_{\overline{z}} \,\textnormal{d} \overline{z}  \,$ gives rise to the interaction of the electromagnetic field with the charged matter in $\,\Delta\,$,   in accordance with \textit{Lorentz force law}. Alexander Khitun's 
 experiments \cite{Khitun} reveal that "\textit{Incorporating  a magnetic field could help electric capacitors store more energy without breaking down}".    For relevant references see \cite{GM, JB, PAT}.



\subsection{Rad\'{o}-Kneser-Choquet Theorem (RKC)}

At this point it is relevant to invoke the classical Rad\'{o}-Kneser-Choquet theorem (RKC) \cite{Rad, Kn, Choquet, AS, IOrkc},  which asserts that \\

\textit{Harmonic extension of a homeomorphism $\, H : \partial \Omega \onto \partial \mathcal Q\,$  of the boundary of a Jordan domain  onto the boundary of a convex domain  takes $\,\Omega \,$  diffeomorphically onto $\,\mathcal Q\,$.}  \\

In a succinct way, the above hypothesis, on  the boundary map $\, H : \partial \Omega \onto \partial \mathcal Q\,$ to be a homeomorphism, forces that $\,\Omega\,$ must be a Jordan domain.  
 A generalization of RKC-theorem, where $\,\Omega\,$ is simply connected, requires additional  topological arguments.  In this case one cannot even speak of a homeomorphism $\, H : \partial \Omega \onto \partial \mathcal Q\,$. Nevertheless  (RKC)-theorem can be generalized if the  boundary map is understood as a monotone map - the concept offered by C.B. Morrey  in 1935  \cite{Morrey}. 
\begin{definition}
 A continuous map $\, H \,:\, \mathbb X  \onto \mathbb Y\,$ between compact topological spaces is said to be monotone if for every $\,y_\circ \in \mathbb Y\,$ its preimage  $\, H^{-1} (y_\circ) \bydef \{ x \in \mathbb X\; :\; H(x)  = y_\circ \}\,$ is connected in $\,\mathbb X\,$. 
\end{definition}
Obviously, homeomorphisms are monotone mappings, see also Remark \ref{MonotonicityRemark} . 

A generalization of Rad\'{o}-Kneser-Choquet Theorem to simply connected domains is presented in \cite{IOsimplyconnectedRKC}. It  includes even $\,p\,$-harmonic deformations.

\begin{theorem} \label{pHarmonicRKC}
 Let $\,\Omega\,$ be a bounded simply connected domain in $\,\mathbb C\,$.
\begin{itemize}
\item   (Existence) Every continuous function $\,H : \partial \Omega \into \mathbb C\,$ admits  unique continuous extension, again denoted by $\, H : \Omega \into \mathbb C\,$, which is $\,p\,$-harmonic in $\,\Omega\,$ , $\,1 < p < \infty\,$.

\item  (Injectivity)  If, moreover, $\,H\,$ takes  $\,\partial \Omega\,$ monotonically onto  the boundary of a convex domain $\,\mathcal Q\,$  then its $\,p\,$-harmonic extension is a diffeomorphism $\, H : \Omega \onto \mathcal Q\,$ . In particular, $\,H\,$ admits no rank-one critical points (Jacobian does not vanish). 

\end{itemize}
\end{theorem}

Loosely associated with the present work are various, far reaching, applications of Theorem \ref{pHarmonicRKC}. The one \cite{IKOdiff} deserves special mentioning.
\subsection{Comments on Dirichlet Energy}

In the hyper-elasticity theory, full  differential $\, \textnormal d H  \bydef  H_z \,\textnormal d  z  \; +  \; H_{\overline{z} }\,\textnormal d \overline{z}\,$ of a complex harmonic function  is referred to as the \textit{deformation gradient}. In this context, it is customary to consider the Dirichlet type integral
\begin{equation} \label{DirichletEnergy}
\mathscr D[H] \bydef  \int_{\Omega\setminus\digamma}  \left (\,|H_z(z)|^2 \,+\, | H_{\overline{z}}(z) |^2\right) \; \textnormal dz
\end{equation}
as  the  \textit{stored energy} of the hyperelastic capacitor. However, minimization of the energy is not always possible under arbitrary boundary conditions. Some boundary data  need not even admit any extension of finite energy. Even if they do admit, the energy-minimal mappings may not necessarily be harmonic , see \textit{Nitsche phenomenon on annuli}\, in Section 21.3 of  \cite{AIMb} and the definite answer in \cite{NitscheConjecture}. \\
That being said,  our questions are concerned 
with harmonic mappings having not only finite but also infinite Dirichlet energy.  The interested reader is referred to \cite{DurenBook, DurenHengartner}   for an exhibition of harmonic mappings in the plane\\

\subsection{Two types of critical points}
We find it useful  to distinguish between two types of critical points of a  $\,\mathscr C^1\,$-mapping $\,H\,: \Omega \setminus \digamma  \rightarrow \,\mathbb C\,$. These are the points where the rank of the differential  $\, \textnormal d H  \bydef  H_z \,\textnormal d  z  \; +  \; H_{\overline{z} }\,\textnormal d \overline{z}\,$    is not maximal,  thus equal to one or zero.  The image of a critical point under such $\,H\,$ is called \textit{critical value}. By Sard's theorem the set of critical values has measure zero.
 \begin{definition} Let $\,H \in \mathscr C^1(\Omega, \mathbb C)\,$.
 \begin{itemize}
 \item A point $\,a\in \Omega\,$ at which the rank of the differential $\, \textnormal d H (z) \,$ equals   1 will be called \textit{rank-one critical point}.
     \item A critical point $\,a\in \Omega\,$, at which the differential  $\, \textnormal d H(z) \,$ vanishes,  will be referred to as \textit{rank-zero critical point}.
 \end{itemize}
 \end{definition}
At a rank-zero critical point we have $\,|\nabla H |^2  \bydef |H_z|^2  + |H_{\overline{z}}|^2  = 0 \,$,  whereas at a rank-one  critical point we have only the Jacobian determinant $\, J_H  \,=\,  |H_z|^2  - |H_{\overline{z}}|^2 \,$ vanishing.

\section {An Example of Rank-one Critical Points} \label{Example}

In this example we are considering  the  punctured  plane $\,\mathbb C_\circ  \bydef \{ z \in \mathbb C ; \; z\not = 0\,\}\,$ and a  harmonic  mapping $ \,H  \bydef \mathbb C_\circ \into \mathbb C\,$ defined by :

\begin{equation} \label{Example}
H = H(z) \bydef  - \lambda \,\log\,|z|^2 \;+ \; z - \frac{1}{\overline{z}} \;,  \;\textnormal{where}\; \,\lambda \bydef  \frac{1}{2} \left( a + \frac{1}{a}\right)\,> 1\,,
\;\; \end{equation}
for some  parameter $\, a > 1\,$ which turns out to be one of the  critical points of $\,H\,$. 
\begin{remark} Let us publicize that it is possible to  construct similar examples for $\,p\,$-harmonic potentials, but here we stick to the case $\, p = 2\,$,  for simplicity.
\end{remark}

Let us take a look at the family of concentric circles $\,\mathcal C_r \bydef \{ z ; |z| = r \,\}\;, r>0\,$. 
Then $\,H\,$  restricted to any of such  circles  is a similarity transformation, 
$$
H(r e^{i \theta}) =  c \,+ \left( r  -\frac{1}{r} \right)  e ^{i \theta}\;,\;\;\textnormal{where} \,\; c \bydef  -  2 \lambda \log\,r\;
$$
In particular,  the  image of $\,\mathcal C_r\,$ under $\,H\,$ is again a circle, centered at $\,c\,$ and with radius $\,\rho \bydef \left|\, r -\frac{1}{r}\,\right| \,$
\begin{equation} \label{Circles}
H(\mathcal C_r) \; \bydef\; \mathcal C_\rho(c) =  \{ w : |w - c | = \rho\}\;
\end{equation} 
There is one exception to this statement; precisely, when the unit circle $\,\mathcal C_1\,$ degenerates  to a puncture at $\,\{0\}\,$, in which case,
\begin{equation}
H(e^{ i \theta}) \, \equiv 0 \;\;
\end{equation}

  It has to be pointed out that the circles $\,\mathcal C_\rho(c)\,$ at (\ref{Circles})  need not be mutually disjoint.\\
 A short computation of complex partial derivatives shows that :
$$ H_z =  1 - \frac{\lambda}{z}\;\;\;\;\textnormal{and}\;\;\;\; H_{\overline{z}} = \frac{1}{\overline{z}^2}  \;-\frac{\lambda}{\overline{z}} $$

\begin{equation}
| H_z|^2 =  1\,-\, \frac{2\lambda}{|z|^2}\, \Re  \, z  \,+\, \frac{\lambda^2}{|z|^2}
\end{equation}

$$
|H_{\overline{z}}|^2 = \, \frac{1}{|z|^4} \,- \frac{2\lambda}{|z|^4}\, \Re  \, z \;+\;\frac{\lambda^2}{|z|^2}
$$
It then follows that the differential does not vanish. Indeed, we have
$$
|z|^2 | H_z|^2 \,+\,|z|^4 |H_{\overline{z}}|^2  = (1 - \lambda )^2 (1 + |z|^2) \; + 2 \lambda \,|1 - z |^2\; > 1
$$
where the strict inequality $\,  > 1\,$ can easily be seen.
Concerning the Jacobian determinant, we compute
$$
|z|^4  J_H(z)  = |z|^4 \left(\,|H_z|^2\,-\, |H_{\overline{z}}|^2    \right)  =  \left(\,|z|^2 - 1  \right) \big[ | z - \lambda |^2  \,-\, ( \lambda ^2 - 1 ) \;    \big]
$$
and recall that $\,\lambda = \frac{1}{2}\left( a +  1/a  \right)\,$, so   $\, |a - \lambda |^2 =  ( \lambda ^2 - 1 ) \, = \left[\frac{1}{2}\left( a - \frac{1}{a} \right)\right]^2\, \bydef \rho^2\,.\,$
Therefore $\,J_H(a) = 0\,$. On the unit circle of $\, |z| = 1\,$, we also have $\, J_H(z)  =  0\,$. There are two remaining locations of the critical points to consider. \\
\textbf{Case 1} (outside the unit circle,  $\, |z| > 1\,$)\\
The Jacobian determinant is positive if and only if $\, |\, z - \lambda\,| > \rho\,$. This is the region outside the circle $\, \mathcal C_\rho (\lambda) \,$ centered at $\,\lambda = \frac{1}{2}\left( a + \frac{1}{a} \right) \,$ and with radius $\, \rho = \frac{1}{2} \left( a - \frac{1}{a} \right)\,$. \\
\textbf{Case 2} (inside the unit circle, $\,0 < |z| < 1\,$)\\
The Jacobian determinant is positive if and only if $\, |\, z - \lambda\,| < \rho\,$.\\
Figure \ref{ExampleCriticalPoints} below illustrates those cases.
\begin{center}
\begin{figure}[h] 
 \includegraphics[angle=0, width=0.90 \textwidth]{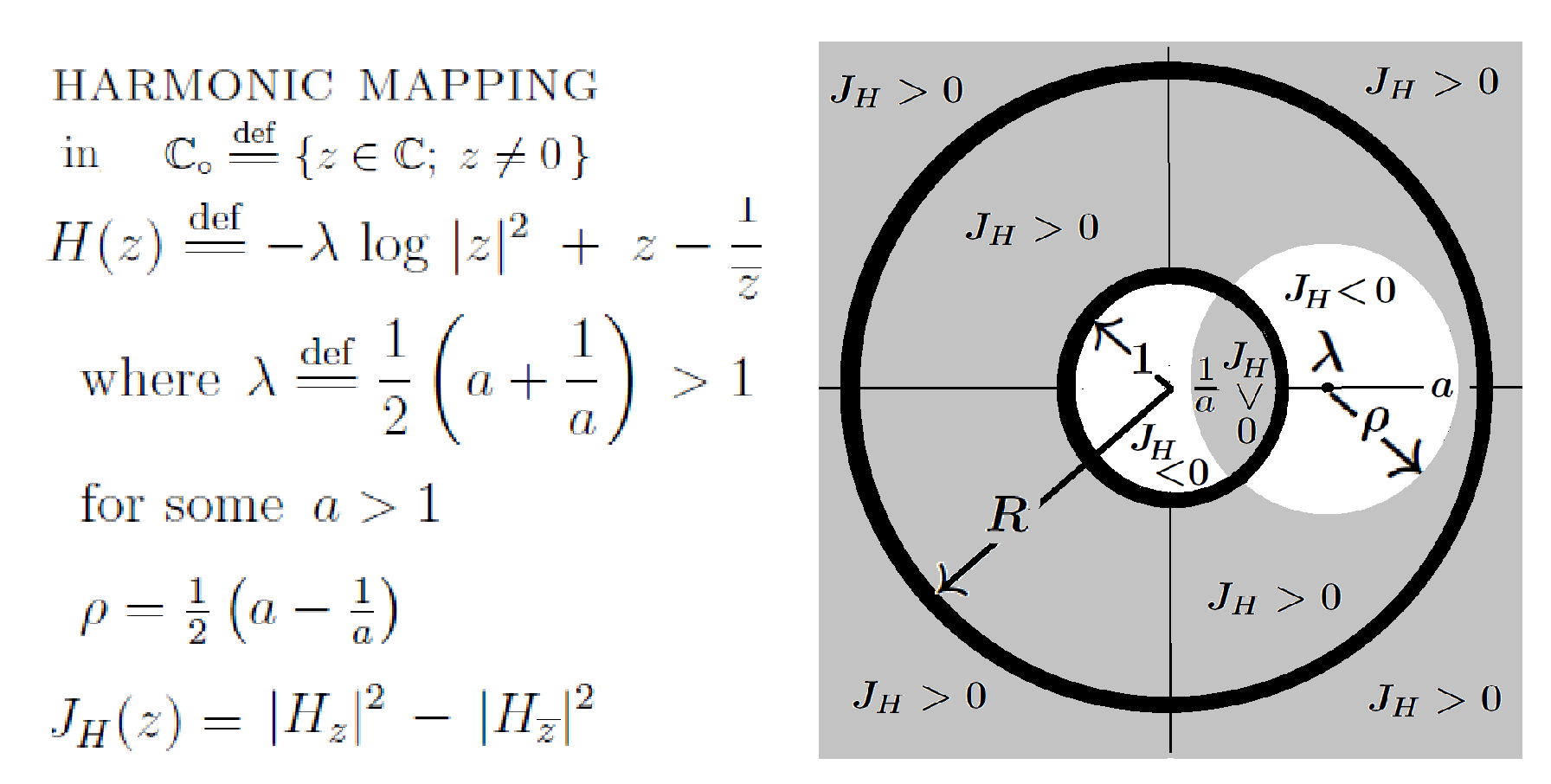}
\caption{  Critical points of $\,H\,$, where the Jacobian determinant vanishes,  lie along the circles $\, |z| = 1\,$ and $\,|z - \lambda| = \rho\,$. However, the rank-zero critical points, where the differential matrix vanishes, cannot occur.} \label{ExampleCriticalPoints}
\end{figure}
\end{center}
\newpage 
\section{Starlike Continua and Monotonicity}

\begin{definition}\label{Starshape}
A starlike continuum $\,\mathfrak G \subset \mathfrak R^2\,$  (in the target plane $\,\mathfrak R^2\,$) is a compact set whose intersection with every straight line passing through a given point $\,\mathfrak s\,$, called a center of $\,\mathfrak G\,$,  consists of at most  two connected components (of course, these components can be: either empty set, single points or closed line segments).    
\end{definition}
\begin{center}
\begin{figure}[h] 
 \includegraphics[angle=0, width=0.90 \textwidth]{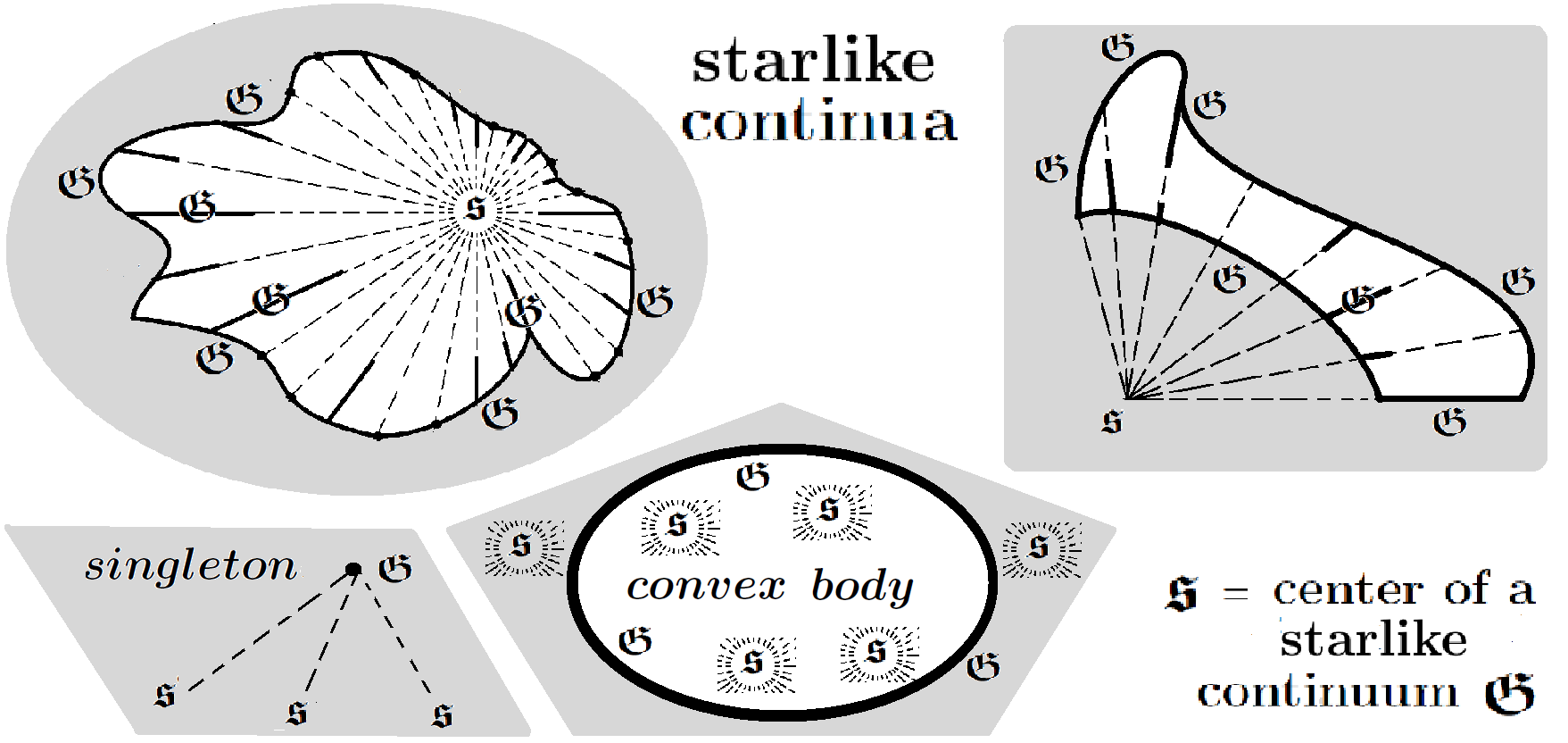}
\caption{  Starlike continum  $\,\mathfrak G\,$ and its possible centers   $\,\mathfrak s . $ } The above starlike continua  are monotone images of  $\,\partial \Omega.$ \label{ExampleStarlikeContinua}
\end{figure}
\end{center}

\begin{remark} \label{MonotonicityRemark}
Monotonicity property of $\,H : \partial \Omega  \onto \mathfrak G\,$  tells us that for every singleton $\, \{\mathfrak c\} \subset  \mathfrak G\,$ its preimage $\,H^{-1}\{\mathfrak c\} \bydef \{ z \in \partial \Omega\; ;\; H(z) = \mathfrak c\,\} $ is connected. Actually, this property implies that the  preimage of any connected set  in $\,\mathfrak G\,$ is connected  in $\,\partial \Omega\,$  as well, see G.T. Whyburn \cite{Whyburn} \,and surveys by L. F. McAuley \cite{Mc, McAuley} .
Concerning monotone Sobolev mappings, we mention \cite{IOmono, IOsimplyconnectedRKC, IOLimitsSobolevHom}  and references therein for further reading.\\
No straight line segment, although still a starlike continuum, can be realized  as a monotone image of $\,\partial \Omega\,$, see Lemma \ref{NoLineSegment} below.
\end{remark} 
\section{Starlike-shaping  of a complex capacitor}
To formulate our main result in the greatest possible generality,  we need to introduce one more definition. As before, $\,\Omega\,$ is a bounded simply connected domain and $\;\digamma  \Subset \Omega\,$  a continuum (not  a single point) that does not disconnect $\,\mathbb C\,$. Thus  $\,\digamma\,$  
 is the bounded component of $\,\mathbb C \setminus \Omega\,$.  
\begin{definition} [Starlike-shaping]  \label{starcondenser}$\;$ \\
 The term\textit{ starlike-shaping of a capacitor}  $\,( \Omega, \digamma)\,$ refers to a  non-constant continuous function $\, H :  \mathbb C \rightarrow \mathbb C\,$  such that
\begin{itemize}
\item $\,H\,$ is harmonic in  $\, \Delta \bydef\; \Omega\setminus \digamma\,$
\item $\,H\,$ is constant on $\,\digamma\,$, say $\,H(\digamma) \,=\, \{\mathfrak s\}$
\item $\,H(\partial \Omega)  \;\bydef\; \mathbf{\mathfrak G }\;$ is a starlike continuum of center at $\, \mathfrak s\,$.
\item The boundary map $\,H : \partial \Omega  \onto \mathfrak G\,$ is monotone.
\end{itemize}
\end{definition}

\section{Main Result}\label{sec:mainresult}

With reference to Definition  \ref{starcondenser}, our main resultt reads as follows. 

\begin{theorem}\label{MainTheorem}
The differential matrix of a starlike-shaping  complex potential $\, H :  \Delta \rightarrow \mathbb C\,$ does not vanish at any point. However, its Jacobian determinant may turn into zero at some points,  see (\ref{Example})\;.
\end{theorem}

That this  implies Theorem \ref {pharmonicGradient} for $\,p = 2\,$,  is the content of the following

\begin{corollary} \label{scalarcase}
Let the  harmonic potential $\,H \,$ of a  capacitor be determined by two real constant boundary values, say $\,H\{\digamma\} = \{0\}\,$  and $\,H(\partial \Omega)  = \{1\}\,$. Then $\,H\,$ is a real-valued harmonic function in $\,\Omega\setminus \digamma\,$. Consequently, its gradient coincides with the  differential matrix and, by Theorem \ref{MainTheorem},  does not vanish at any point.
\end{corollary}
\section{Harmonic Dendrite }

Our proof of  Theorem \ref{MainTheorem} is based on  analysis of  a scalar harmonic function   $\, W(z) \bydef \alpha\,U(z) \, + \,\beta\, V(z)\,$. This  linear combination of the real part $\,U\,$ and the imaginary part $\,V\,$  of the complex harmonic potential  $\, H = U + i V  : \, \Delta \bydef \Omega\setminus\digamma \into \mathbb C\,$ will be  specified later.  The right choice of the coefficients $\,\alpha, \beta\, ,\,  \alpha^2 + \beta^2 \not= 0\,$ is crucial. We shall examine a connected component of the level set $\, \{z \in \Delta\;:\;   W(z) = \textnormal{constant} \,\}\,$, where the constant will also be chosen  accordingly.     
We follow the ideas in \cite{IOsimplyconnectedRKC}; particularly, from Section 7 therein.  These ideas, the results, and comments are worth the efforts of discussing them further. \\

Let us temporarily consider  an arbitrary nonconstant harmonic function $\, W : \Delta \,\into \mathbb R\,$ .   Since the complex gradient $\, W_z = \partial W/ \partial z\,$ is holomorphic, the set of critical points $\,\mathbf C \bydef\{ z ;  \nabla W(z) = 0\}\,= \{ z_1, z_2, ..., ...\}$ is discrete. Outside $\,\mathbf C\,$   each level set $\,\mathfrak L \bydef \{ z \in \Delta\,;  W(z) = \textnormal{constant} \,\}\,$  consists of mutually disjoint $\,\mathscr C^\infty\,$-smooth open arcs  whose endpoints are the critical points. However, some  arcs  may not have endpoints; instead, they may terminate with so-called \textit{limit sets}. Those limit sets are continua in  $\,\partial \Delta\,$, see Lemma  \ref{DisjointLimitSets} and Figure \ref{CriticalDendrite} .

\subsection{Local structure of a harmonic level set}\label{LocalStructure}  We appeal to Corollary 10 in \cite{IOsimplyconnectedRKC} from which we  extrapolate the following:
\begin{itemize}
\item  Consider a harmonic level set $\,\mathfrak L\,$ passing through a point $\, a \in \Delta\,$. Then near this point $\,\mathfrak L\,$ consists of  $\,2m\,$ Jordan sub-arcs emanating from $\,a\,$ in the directions that (upon a suitable rotation) are just the unit complex numbers $\, \textnormal {Exp} \left(\frac{2k -1} { \,2m\,}  \pi\,i  \right)\;,\; k = 1,2, ... , 2m   \,$.
    \item Here $\, m = n+1\,$ where $\,n \geqslant 0 \,$ stands for the order of zero of the holomorphic function $\, W_z\,$ at $\,a\,$. In particular, we have $\, 2m = 2\,$ if $\,a\,$ is a regular point and $\, 2m \geqslant 4\,$  if $\,a\,$ is a critical point of $\,W\,$.
\end{itemize} 

\begin{corollary}{$\,\mathfrak L\,$ is locally path connected}
\end{corollary}
\begin{proof}This conclusion is immediate from the  local structure of $\,\mathfrak L\,$. \end{proof}
It does not mean, however,  that $\,\mathfrak L\,$ is connected. For this reason we now restrict our discussion to a connected component of $\,\mathfrak L\,$ that contains a given point $\, a \in \Delta\,$  and  reserve the notation $\,\mathfrak L_a\,$ for the above  connected component; primarily,  $\,a\,$ will be a critical point of $\,W\,$.
 $\,\mathfrak L_a\,$ is the connected component of $\, \{z \in \Delta\;:\;   W(z) = W(a) \,\}\,$ that contains $\,a\,$.  Call it 
 a \textit{harmonic  dendrite} of $
\,W\,$ passing through its critical point $\,a\,$. 
\begin{center}
 \begin{figure}[h]
 \includegraphics[angle=0, width=0.79 \textwidth]{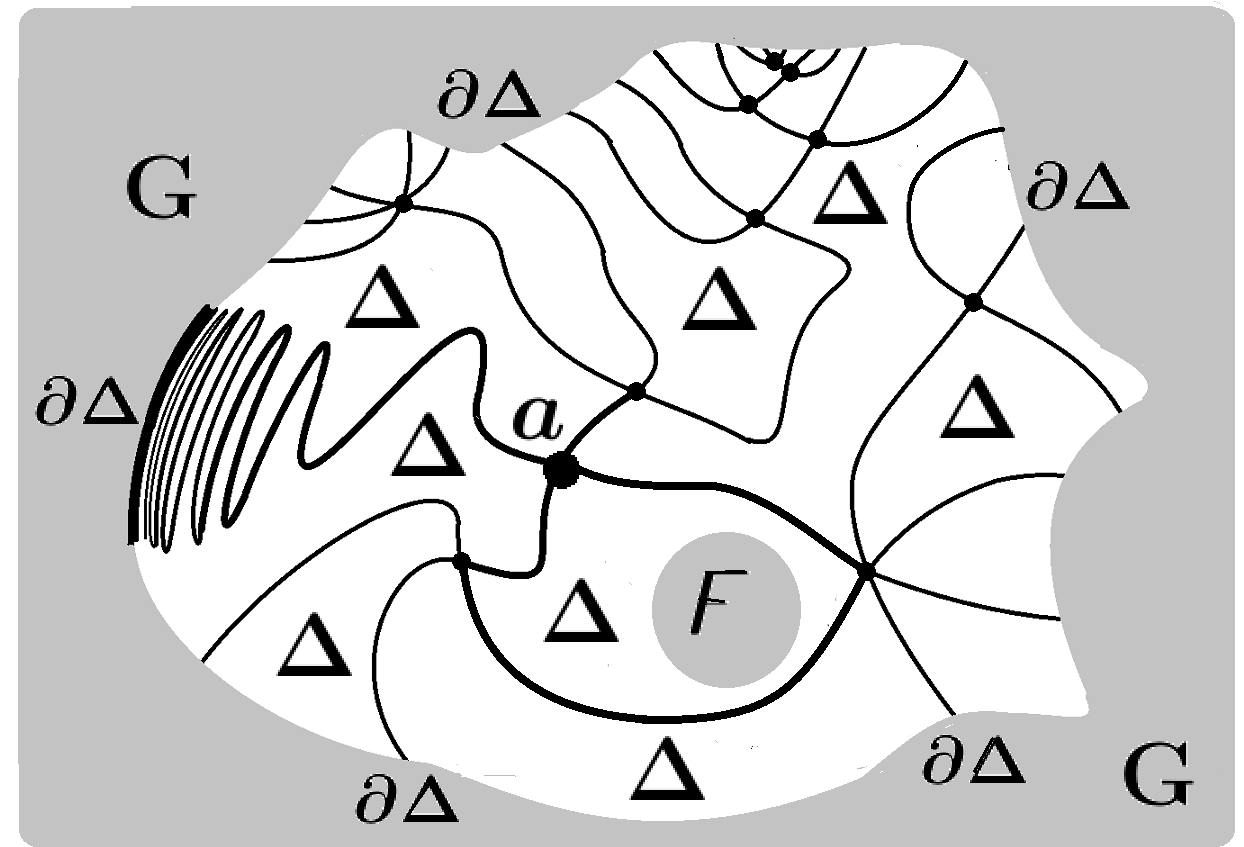}
\caption{ Harmonic dendrite $\,\mathfrak L_a\, \bydef \,\{ z \in \Delta \,;\, W(z) = W(a)\}\,$ passing through a critical point  $\, a \in \Delta = \Omega\setminus\digamma \,$. .}  \label{CriticalDendrite}
\end{figure}
\end{center}
\begin{lemma}
Any closed loop in $\,\mathfrak L_a\,$ must surround the hole $\,\digamma \Subset \Omega\,$, see Figure \ref{CriticalDendrite} .
\end{lemma}
\begin{proof}
Because otherwise such a loop would constitute the boundary of a subdomain of $\,\Delta\,$ along which $\,W  = \textnormal{constant}\, = W(a)\,$. Consequently $\,W(z) = W(a) \,$ inside the loop. By the unique continuation property of harmonic functions, we would conclude that $\,W (z) \equiv W(a)\,$ in the entire domain $\,\Delta\,$; contrary to the assumption that $\, W \not = \textnormal{constant}\,$. 
\end{proof}
In much the same way we infer that the closure of $\,\mathfrak L_a\,$, which might include points in $\,\partial \Delta\,$, cannot contain the boundary of a subdomain of $\,\Delta\,$. Figure \ref{BranchesMeetingOnTheBoundary} illustrates such an impossible circumstance. In this illustration  we have a subdomain $\,\Delta^0 \subset \Delta\,$ (dark shaded sinusoidal region)  whose boundary is contained in  $\,\overline{\Delta}\,$. Recall that $\,W\,$ is assumed to be continuous up to $\,\overline{\Delta}\,$, so $\,W \equiv 0\,$ on   $\,\Delta^0\,$ and, therefore, on the entire region $\,\Delta\,$ as well; which is a clear  contradiction.

\begin{center}
 \begin{figure}[h]
 \includegraphics[angle=0, width=0.99 \textwidth]{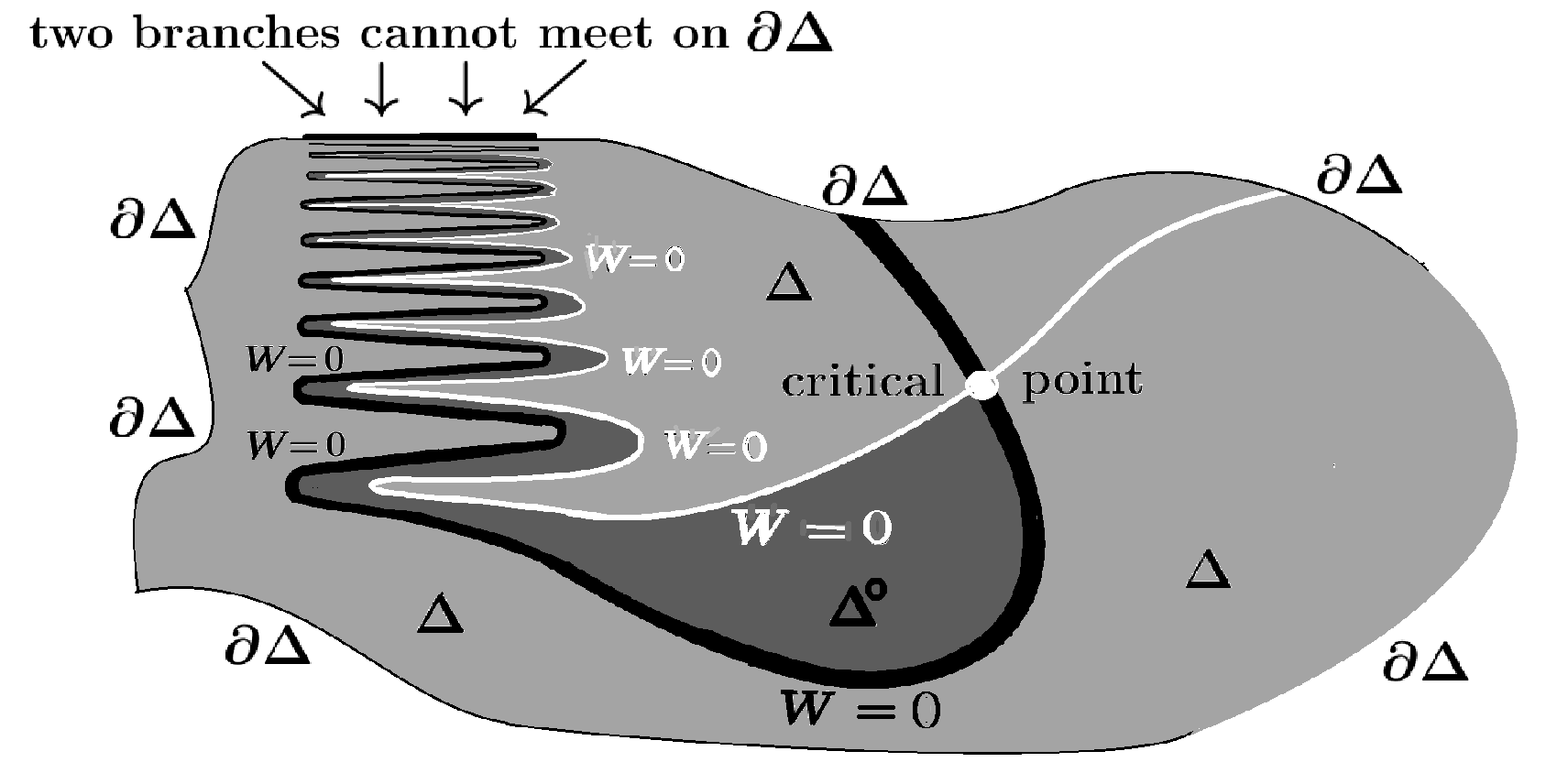}
\caption{ Two different branches of the harmonic dendrite $\,\{ z \in \Delta \,;\, W(z) = 0\}\,$ emanating from a critical point of $\,W\,$ never meet again, neither in $\,\Delta\,$ nor on $\,\partial \Delta\,$.  }\label{BranchesMeetingOnTheBoundary}
\end{figure}
\end{center}
From now on, while  loosing no generality, we assume (as in the above figure) that:
\begin{equation} \, W(a) = 0 \, \textnormal{\,and\,} \, W(\digamma)= \{0\,\}\,
\end{equation}

\begin{corollary} There are no closed loops in  $\,\mathfrak L_a\,$.   
\end{corollary}

\begin{proof}
We need only rule out  the presence of closed loops  in  $\,\mathfrak L_a\,$  that surround $\,\digamma\,\subset \Omega$.
 Such a loop would constitute the boundary of a subdomain, say  $\,\Omega^0 \Subset \Omega\,$ and  $\,\digamma  \Subset \Omega^0 \Subset \Omega\,$.  This would give us a doubly connected region $\,\Omega^0 \setminus \digamma\,\subset \Delta\,$ and a harmonic function $\,W\,$ vanishing on both components of its boundary. Thus, as before,  $\,W \equiv 0\,$ on  $\,\Omega^0 \setminus \digamma\,$ and, by the unique continuation property, $\,W \equiv 0\,$ on $\,\Delta\,$.
\end{proof}

\subsection{Jordan paths $\,\mathscr P_1,\, \mathscr P_2, \mathscr P_3, \,\mathscr P_4, ...\,$, emanating from the critical point $\,a\,$} 
 We initiate our construction of the paths $\,\mathscr P_1,\, \mathscr P_2, \mathscr P_3\, ...\,$, from the critical point, denoted by $\, z_1 \bydef a\,$.  According to Subsection \ref{LocalStructure},  there emanate from $\, z_1\,$ at least four local arcs along the dendrite $\,\mathfrak L_a\,$, say  $\,\ell_1, \ell_2, \ell_3, \ell_4, ...\,, \ell_{2m}\,$\,, $\, m \geqslant 2\,$, for which  $\,z_1\,$ is their only common point (call it their left endpoint). 
Except for $\,z_1\,$,  these arcs  remain mutually disjoint, because of no presence of closed loops in $\,\mathfrak L_a\,$.  We prolong each arc $\,\ell_\nu\, ,\, \nu = 1, 2, ... , 2m\,$,   until we  reach another  critical point, say $\,z_1^\nu\,$    and view it as right endpoint of $\,\ell_\nu\,$ . It is possible, however, that no prolongation terminates at any critical point. In that case the arc must approach $\,\partial \Delta\,$, due to local structure of the dendrite. Next,  each new critical point  $\,z_1^1,\, z_1^2, ...\, z_1^{2m}\,$ plays  the role of the left endpoint of a new arc emanating from it.  For each $\, \nu = 1, 2, ..., 2m\,$ we choose and fix one of those arcs and denote it by  $\,\ell_1 ^{\,*},\,  \ell_2^{\,*}, ...\,, \ell_{2m}^{\,*}  \,$, respectively. As before, each of these new arcs prolongs untill it reaches a critical point or  approaches $\,\partial \Delta\,$.  Every such critical point becomes a left endpoints of a finite family of arcs emanating from it.  We choose one of them and denote the chosen arcs by  $\,\ell_1 ^{\,**},\,  \ell_2^{\,**}, ...\,, \ell_{2m}^{\,**}  \,$, respectively. Similarly,  we construct a third generation of arcs $\,\ell_1 ^{\,***},\,  \ell_2^{\,***}, ...\,, \ell_{2m}^{\,***}  \,$, and so on.  Continuing in this fashion (indefinitely if necessary) we obtain $\,2m\,$ of Jordan paths emanating from $\,z_1 = a\,$ and approaching the boundary of $\,\Delta\,$. In symbols, 
$$
\mathscr P_\nu  \bydef \ell_\nu ^{\,*} \cup \ell_\nu^{\,**} \cup \ell_\nu^{\,***} \,\cup\,...  
$$
 Precisely, $\,\emptyset \neq \overline{\mathscr P_\nu }\setminus \mathscr P_\nu  \subset \partial \Delta\,$, for $\,\nu = 1,2, ... , 2m\, , \, 2m \geqslant 4\,$.  Since the entire dendrite $\,\mathfrak L_a\,$ admits no closed loops, we have
$$
\overline{\mathscr P_\nu } \cap \overline{\mathscr P_\mu}\,=\, \{a\}\,\,\,\,\,,\,\, \textnormal{for all}\; 1 \leqslant \mu \,\neq\, \nu\, \leqslant 2m 
$$
Also, $\,\mathscr P_\nu\,$ does not intersect itself. 

\begin{lemma}\label{DisjointLimitSets}
The so-called limit-sets, or the right end-sets,
$$ \mathcal C_\nu   \bydef \overline{\mathscr P_\nu }\setminus \mathscr P_\nu  \subset \partial \Delta\,,\;   \textnormal{for} \; \nu = 1,2, ... , 2m\, ,\;\textnormal{where} \, \;2m \geqslant 4\,,$$
are mutually disjoint continua in $\,\partial \Delta\,$\;\textnormal{(not like in Figure \ref{BranchesMeetingOnTheBoundary})}.
\end{lemma}
\begin{proof} To see that the right end-set $\overline{\mathscr P_\nu }\setminus \mathscr P_\nu  \subset \partial \Delta\,$  is indeed a continuum, we parametrize $\,\mathscr P_\nu\,$ via a homeomorphism $\, \phi :\, [0,\; \infty ) \, \onto \mathscr P_\nu\,\,,\; \phi(0) = a\;$. Thus $\,\mathcal C_\nu\,$ is the set of all possible limits of $\,\phi(t)\,$ as $\,t\,$ approaches  $\infty\,$. Consider the decreasing sequence of left-truncated paths
$$\,\mathscr P_\nu^0 \,\supset \,\mathscr P_\nu^1\, \supset\mathscr P_\nu^3 \,\supset ...   \;\,,\; \textnormal{where} \;\,\mathscr P_\nu^n\; \bydef \phi[n,\; \infty )\;,\; n = 0,1,2, ... $$
They are connected, so are their closures. Therefore, the limit-set $\,\mathcal C_\nu\,$, being the intersection of a decreasing sequence of continua, is a continuum as well. In symbols, 

$$\,\overline{\mathscr P_\nu^0} \,\supset \,\overline{\mathscr P_\nu^1}\, \supset\overline{\mathscr P_\nu^3} \,\supset ...  \;\;\textnormal{and} \;\;\bigcap_{ n \geqslant 0} \overline{\mathscr P_\nu^n} \; =\;\mathcal C_\nu\;,$$
as desired. \\
Now suppose, to the contrary, that two limit-sets $\,\mathcal C_\mu \,$ and   $\,\mathcal C_\nu \,$ , \,$\,\mu \not =\nu\,$, intersect. Thus we have two continua $\,\overline{\mathscr P_\nu}\,$  and  $\,\overline{\mathscr P_\mu}\,$  whose intersection is not connected; namely, it is a union  of two disjoint compact sets
$$
\overline{\mathscr P_\nu}\,\cap\,\overline{\mathscr P_\mu}\, =\, \big[\{a\}\big]\; \cup \; \big[\mathcal C_\mu \, \cap \mathcal C_\nu \,\big]\,\,
$$
In general the set $\,\big[\mathcal C_\mu \, \cap \mathcal C_\nu \,\big]\,$ can be very bizarre. Figure \ref {BranchesMeetingOnTheBoundary} \; illustrates  two paths $\,\overline{\mathscr P_\nu}\,$ (in black) and  $\,\overline{\mathscr P_\mu}\,$ (in white) that are approaching the same limit-set on $\,\partial \Omega\,$ , in a rather clear manner. However, for more complicated configurations of paths one needs to appeal to an elegant topological result established in 1913 by S.  Janiszewski \cite{Ja} , \cite{Kub}. 
\begin{theorem}  [Janiszewski] Let $\,\Gamma_1\,$ and $\,\Gamma_2\,$ be continua in $\,\mathbb R^2 \,$ whose intersection $\,\Gamma_1\,\cap\,\Gamma_2\,$is not connected. Then their union $\,\Gamma_1\,\cup\,\Gamma_2\,$ disconnects $\,\mathbb R^2\,$. 
\end{theorem}
Now we see that all bounded complementary components of the union $\,\overline{\mathscr P_\nu}\,\cup\,\overline{\mathscr P_\mu}\, \,$  lie in $\,\Delta\,$.  The harmonic function $\,W\,$ , turning to zero on the boundary of such components, must  vanish inside the components. By unique continuation property $\, W \equiv 0\,$ in $\,\Delta\,$, which is the desired contradiction.
\end{proof}

\section{Proof of Theorem \ref{MainTheorem}}
With the above preliminaries at hand the proof of Theorem \ref{MainTheorem} is based on the ideas and geometric arguments about harmonic dendrite of a  linear combination  $\,W(z) = \alpha \,U(z) \,+ \beta \, V(z)\,$ discussed above. The coefficients $\,\alpha, \beta\,$ are selected specifically for the given continuum $\,\digamma\Subset \Omega\,$ and the critical point  $\,a \in \Delta = \Omega \setminus \digamma\,$. 

\begin{proof} Suppose to the contrary, that $\,a\,$ is a rank-zero critical point of $\,H = U + i V\,$ ; that is, $\,\nabla U(a) = 0\,$ and $\,\nabla V(a) = 0\,$. Recall that $\,H : \partial \Omega \onto \mathfrak G\,$ is a monotone mapping onto a starlike continuum $\,\mathfrak G\,$. The so-called center of $\,\mathfrak G\,$ is located at the point $\,\mathfrak s \,$ defined by $\,   \{\mathfrak s\} =  H(\digamma)\,$. We may, and do, assume that $\,\mathfrak s = 0\,$, which is the origin of the coordinate system in the target plane $\, \mathfrak R^2 =\{(u,v)  ;\; u,\,v \in \mathbb R\,\}\,$.  Denote the image of the critical point by $\,\mathfrak a \bydef H(a)\,=  \mathfrak u \,+\,i \,\mathfrak v \, = U(a) \,+ i \,V(a)\,$. Now the real coefficients $\,\alpha\,,  \beta\,$ are chosen to satisfy : $\, \alpha^2 + \beta^2 = 1\, \,\textnormal{and} \;\,\alpha \,U(a) \; +\; \beta\, V(a) \,= 0\,$; that is,  $\,(\alpha ,\beta ) \,$ is the unit vector orthogonal to $\,\mathfrak a\,$. Note that $\,(\alpha, \beta)\,$  is determined uniquelly up to  its direction, except for the uncommon occasion of $\,\mathfrak a = \mathfrak s = 0\,$, in which case any unit vector serves the purpose.   We examine the following harmonic function in $\,\Delta\,$ that is continuous up to $\,\overline{\Omega}\,$, 
$$
W(z) \bydef \alpha\,U(z) \,+\, \beta\,V(z)\;,\; \textnormal{thus}\; W(a) = 0\;\;\textnormal{and}\;  W(z) \,\equiv 0\,\, \textnormal{in}\; \digamma\,.
$$
Since $\,a\,$ is assumed to be a rank-zero critical point of $\,H\,$, we have $\,\nabla W(a) = 0\,$, regardless of the choice of the coefficients $\,\alpha ,  \beta\,$ . \\
Next denote by $\,\mathfrak X \,$  the straight line  in the target plane \\$\,\mathfrak R^2 \bydef \{ (u,v) ;\; u , v \in \mathbb R \}\,$ passing through the point $\,\mathfrak a =  H(a) \bydef \mathfrak u \,+\,i \,\mathfrak v \, = U(a) \,+ i \,V(a)\,$ and the center $\,\mathfrak s= 0\,$ of the star-shape continuum $\, H(\partial \Omega)  = \mathfrak G\,$. According to Definition  \ref{Starshape}  the intersection 
\begin{equation}\label{LimitContinua}
\mathfrak X \,\cap\, \mathfrak G \;\;\textnormal{consists of at most two disjoint continua.}
\end{equation}

Their preimages  under the  monotone map $ \, H : \partial \Omega  \,\onto\, \mathfrak G\,$ are disjoint continua in $\,\partial \Omega\,$. In fact, these are  limit-sets of  the dendrite  $\,\mathfrak L_a\,$ that lie in $\partial \Omega\,$. The  remaining limit-sets of $\,\mathfrak L_a\,$ lie in $\,\partial \digamma\,$.  At the same time, according to Lemma  \ref{DisjointLimitSets}  , we have at least four Jordan paths $\,\mathscr P_1, \mathscr P_2, \mathscr P_3, \mathscr P_4,  (\, ...\,, \mathscr P_{2m})\,$ emanating from the critical point $\,a\,$ and terminating in mutually disjoint continua  $\,\mathcal C_1, \mathcal C_2, \mathcal C_3, \mathcal  C_4, (\, ...\,, \mathcal  C_{2m})\,$,  which are exactly the limit-sets of $\,\mathfrak L_a\,$. Only two of them may lie in $\,\partial \Omega\,$, because of (\ref{LimitContinua}). Thus we have at least two paths, say $\,\mathscr P_\nu\;,\, \mathscr P_\mu\,\;, 1 \leqslant \nu \not= \mu \leqslant 2m \,,\,$ which terminate in $\,\digamma\,$ (recall,  there are no closed loops in $\,\mathfrak L_a\,$). This situation is illustrated in Figure \ref{SophisticatedPaths}.
\begin{center}
 \begin{figure}[h]
 \includegraphics[angle=0, width=0.99 \textwidth]{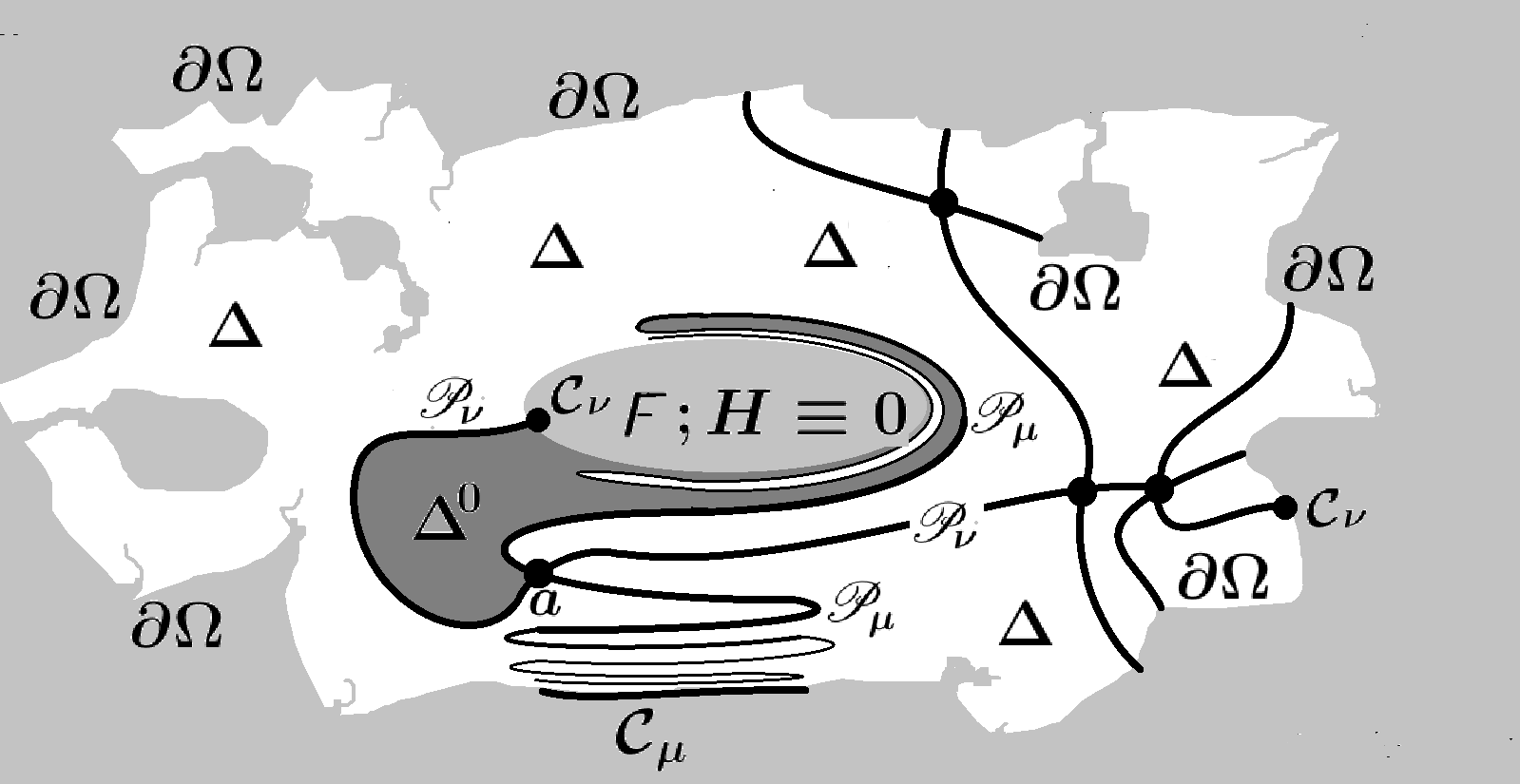}
\caption{ Each of the two  paths $\,\mathscr P_\mu\,$ and $\,\mathscr P_\nu \,$,  passing through a critical point $\,a \in \Delta\,$, connects $\,\partial \digamma\,$ and $\,\partial \Omega\,$ via its limit end-sets $\,\mathcal C_\mu\,$ and $\,\mathcal C_\nu\,$, respectively. }\label{SophisticatedPaths}
\end{figure}
\end{center}A concerned reader can dispense of this illustration and easily  argue (under the most general  circumstances)  by applying S. Janiszewski theorem \cite{Ja}  again.  For this application we look at  two continua $\, \Gamma_\nu = \overline{\mathscr P_\nu} \cup \digamma\,$  and $\, \Gamma_\mu = \overline{\mathscr P_\mu} \cup \digamma\,.$ Their intersection is not connected, $\, \Gamma_\nu \cap \Gamma_\mu \; =\; \{a\}\; \cup \,\textnormal{a  subset of}\; \digamma\,$ . Therefore,  their union $\, \Gamma_\nu \,\cup \,\Gamma_\mu\,$ disconnects $\,\mathbb R^2\,$. The bounded complementary component of  $\, \Gamma_\nu \cup \Gamma_\mu\,$ , denoted in Figure  \ref{SophisticatedPaths}\ by $\,\Delta^0\,$ , lies in $\,\Delta\,$ . As a consequence, the harmonic function $\,W\,$ (being equal to zero on $\,\partial \Delta^0\,$ )  must vanish in $\,\Delta^0\,$ and hence (by unique continuation property) on $\,\Delta\,$ as well.  This means that the non-constant map $\,H\,$ takes the entire boundary $\,\partial \Omega\,$  onto a continuum in $\, \mathfrak X\,$; that is, onto a straight line segment. 
Now the proof of  Theorem \ref{MainTheorem} is completed by the following simple observation.
\begin{lemma} \label{NoLineSegment}There is no monotone map of $ \partial \Omega\,$ \,onto a line segment.
\end{lemma}
Indeed, such a monotone map, say  $\,H : \partial \Omega\, \onto [A , B]\,$,  would give us two disjoint continua $\,H^{-1}(A)  \subset \partial \Omega\,$  and $\,H^{-1}(B)  \subset \partial \Omega\,$. At the same time the  preimage $\,H^{-1}(A,B)\,$ of the remaining open interval would also be connected, see  \cite{Whyburn} or, for a simple proof, apply Lemma 20.1.3. in \cite{AIMb}  to closed subintervals of $\,(A, B)\,$ . However, the complement in $\partial \Omega\,$ of the union of two disjoint continua in $\partial \Omega\,$ cannot be connected. 
\end{proof}

\section{Proof of Theorem  \ref{pharmonicGradient} \, (static $\,p\,$-harmonic capacitor)} \label{Static-p-harmonicField}
The proof makes heavy use of  the local structure of the level set of a scalar $\,p\,$-harmonic function $\,u : \Delta \into \mathbb R\,$ near its critical point, as  presented in \cite{IOsimplyconnectedRKC}, see closely related earlier work \cite{AS, BI, IM, Alessandrini,  AL, IKO} .   \\

We may, and do, assume that the constant inner boundary value $\, C_{\textnormal{inn} } \,$, $\,\{ C_{\textnormal{inn }}\} \bydef u(\partial \digamma)\,$,  is smaller than that of the  outer boundary value $\,C_{\textnormal{out} } \,$,  $\,\{ C_{\textnormal{out}}\} \bydef u(\partial \Omega)\,$. \\
First we note, using the \textit{strong max/min  - principle} , see Section 6.5 of the monograph \cite{HKMb}     that  

\begin{equation} \label{MaxMinPrinciple}
 C_{\textnormal{inn }} < u(a)  \,<  C_{\textnormal{out}}\;, \; \textnormal{for every point }\, a \in \Delta\,.
\end{equation}

The crucial  observation, however,  is that  the complex gradient  $\, u_z  \bydef  \partial u / \partial z\,$  is a  nonconstant $\, K\,$ -quasiregular map \cite{BojIwanPreprint, BI, AL, IOsimplyconnectedRKC}\,. Its maximal distortion equals $\,K = \max\{p-1\,,\, 1/ p-1\}\,$.  Thus $ u_z\,$  admits only isolated zeros. 
Suppose, contrary to the assertion of Theorem \ref{pharmonicGradient} , that  $\, u_z \,$  vanishes at some  critical point  $\, a \in \Delta\,$; in symbols, $\,u_z(a) = 0\,$. Observe that the level set $\, \{z \in \Delta\,;\; u(z) = C \bydef u(a)\,\}\,$ is compact. This is because of (\ref{MaxMinPrinciple}). Such a level set  is generally not connected. Therefore, we restrict  further  discussion to  its connected component  that contains the critical point $\,a\,$. Denote it by  $\,\mathfrak D = \mathfrak D_a\,$ and name it \textit{critical $\,p\,$- harmonic dendrone} emerging from $\, a\,$.  We just reached  a branchlike topological structure that requires detailed analysis. For the sake of clarity, let us specify some mathematical terms. 
\begin{itemize}

\item A Jordan closed arc is an image of  a homeomorphism $\, f : [ 0, 1] \rightarrow \mathbb R^2\,$  (continuous bijective mapping). 

\item A Jordan curve (also called Jordan loop)  is the image of a  continuous mappings $\,f: [0,1] \rightarrow \mathbb R^2\,$  which satisfies:  $\,f(0) = f(1)\,$,  and the restriction of $\,f\,$ to the half open interval $\,[0,1)\,$ is injective. 

\item  One might find it convenient to look at the  endpoints of $\,J \bydef f[0,1]\,$; one is  $\,  J^-\, \bydef\,f(0)\,$, and the other $\,  J^+\,\bydef\,f(1)\,$.  In case of the Jordan curve (closed loop) we have   $\,  J^-\, =\, J^+ \bydef J^\mp\,$. 
\item  The image $\,f(0,1)\,$ of the remaining open interval  is  a Jordan open arc called the interior of $\,J\,$ and denoted by $\, J^\circ\,$. 
One must be aware that not every homeomorphic image of $\,(0\,,\,1)\,$ can be obtained that way; its closure may neither be a Jordan closed arc nor a  Jordan loop. 
\item Let us give a common name \textit{Jordan contours} to both Jordan closed arcs  and  Jordan loops, and write
$$
J = J^- \cup J^\circ \cup J^+\;,  \; \textnormal{also }\, \; J \bydef  [J^-, J^+]\,,
$$
whenever clarity requires that the endpoints be indicated.
\end{itemize} 
Here are some properties of the dendrone $\,\mathfrak D \,$ that are characteristic of a graph.

\begin{itemize}
\item $\,\mathfrak D\,$ is a finite union of Jordan contours, 
\begin{equation}\label{ContoursOfDendrone}
\mathfrak D  = {J_1} \,\cup\,{J_2} \,  \cup ... \cup {J_N} 
\end{equation}
\item Their interiors $\,  J_1^\circ \,,\, J_2^\circ\,,  ... , \,J_N^\circ \, $ are mutually disjoint.  
\item Their endpoints $\,J_1^{\pm} \,,\,\,J_2^{\pm} , ... , \,J_N^{\pm}\,$ are critical points of $\,u\,$, also referred to as  \textit{nodes} of the dendrone. \\
 At this stage we appeal to Corollary 10 in \cite{IOsimplyconnectedRKC}. to infer that:
\item From every node  there  emanate an even number, say   $\,2 m\, \geqslant 4\,$, of closed Jordan subarcs  of $\,\mathfrak D\,$. The intersection of any two of them is just that node.  The case where two such subarcs are  parts of  the same Jordan contour $\,J\,$ is not excluded, which happens when $\,J\,$ is a Jordan loop and its endpoints $\,J^-\,,\, J^+\,$ coincide with that node.  
\end{itemize}
The above properties of a  dendrone fit together into the general scheme of a planar graph.  This concept emerges most clearly when the setting is quite abstract. \\

Accordingly, a graph is a pair $\,\mathfrak G = ( V, E )\,$ of  a nonempty set $\,V\,$  (called the vertices) and a set $\,E\,$   of two-element subsets of $\,V\,$  (called the edges) . Now we view the $\,p\,$-harmonic dendrone $\,\mathfrak D\,$ as a graph in which the Jordan contours  $\,{J_1} \,,\,{J_2} \,,\, ... \,, {J_N} \,$ play the role of edges, whereas  their endpoints  are the vertices (nodes of $\,\mathfrak D\,$). The graph of a dendrone is called \textit{planar} because it is connected and can be  drawn without edges crossing (precisely, the interiors $\,  J_1^\circ \,,\, J_2^\circ\,,  ... , \,J_N^\circ \, $ are mutually disjoint). 
In Graph Theory, the number $\, \textnormal{deg } v\,$  of edges emanating from  a given vertex $\,v \in V\,$   is called the degree of that vertex. In case of our dendrone at every node  the degree is an even number, say   $\,2 m \geqslant 4\,$.  The so-called \textit{Handshake Lemma}  tells us that the sum of the degrees of vertices is always twice the  number of edges; in symbols, 

\begin{equation} \label{HandShakeLemma} 
\sum_{v \in V}  \textnormal{deg } v\,  = \, 2\, | E | \;. \;   \textnormal{Hence } \, \; N = |\, E \,| \geqslant 2 |\,V\,|\,.  
\end{equation} 
 Every planar graph divides the plane into regions called faces  (open connected sets). We let  $\,F\,$ denote the set of faces, including the one unbounded region. 
Euler's  celebrated formula  for connected planar graphs reads ad:

\begin{equation} \label{EulerFormula} 
|\,F\,|  = |\,E\,|\, -\, |\,V\,|\, + 2
\end{equation}
 Returning to our $\,p\,$-harmonic dendrone $\,\mathfrak D\,$, the complement  $\,\mathbb R^2 \setminus \mathfrak D\,$ consists of a finite  number of  so-called complementary components. Each component is surrounded by a number of Jordan contours from  (\ref{ContoursOfDendrone}). These are exactly the faces of the graph of $\,\mathfrak D\,$. Our ultimate goal is to demonstrate, using Euler's formula,  that
\begin{lemma} \label{ComplementsOfDendrone}
There are at least two bounded complementary components of $\,\mathfrak D\,$. 
\end{lemma}

  \begin{proof}  
  Having disposed of Euler's formula (\ref{EulerFormula}) and  Equation (\ref{HandShakeLemma}), the proof of this lemma is a matter of straightforward estimate  of the number of all complementary components, including the unbounded one. 
\begin{equation} \label{EulerFormula} 
|\,F\,|  = |\,E\,|\, -\, |\,V\,|\, + 2 \; \geqslant \, 2 |V\,|  \,-\,|\,V\,| \,+2  \,= \,|\,V\,| \,+2 \, \geqslant\, 3
\end{equation}
as desired. 
\end{proof}
We shall now establish the proof of Theorem  \ref{pharmonicGradient} by noticing that each bounded component of $\,\mathbb R^2 \setminus \mathfrak D\,$ must contain the hole $\,\digamma \Subset \Omega\,$. Otherwise, such a component would lie entirely in $\,\Delta = \Omega \setminus \digamma\,$ \,( because $\,C_{\textnormal{inn }} < u(a) \,$  ) and $\, u\,$ would be constant in this component. This means that $\, u\,$ would be constant in $\,\Delta\,$ (by unique continuation) , which is ruled out. On the other hand, we have at least two bounded components of  $\,\mathbb R^2 \setminus \mathfrak D\,$. They are disjoint, by the very definition. Finally, we arrive at the required contradiction  since  $\,\digamma\,$, being a continuum,  cannot lie in two different components. \\
The proof of Theorem \ref{pharmonicGradient} is complete.

\section{Cassini Capacitor with Three  Electrical Conductors} 

Most of the natural capacitors indeed consist
of two electrical conductors.  Here we discuss an example of  a harmonic \textit{dielectric medium}  within three conductors; that is, in $\,\Delta = \Omega \setminus (\digamma_+\cup \digamma_-\,),$ where $\,\Omega\,$ is a bounded simply connected domain and $\,\digamma_+ \Subset \Omega\,$,\, $\,\digamma_- \Subset \Omega\,$ \,are disjoint continua (holes in $\Omega\,$).  The elecrostatic potential $\, u  : \Delta  \into \mathbb R\,$  assumes constant vales on each of three conductors : $\, u \equiv c_0 \,$ on $\,\partial \Omega\,$ , $\, u \equiv c_+ \,$ on $\,\partial \digamma_+\,$ \, and\, $\, u \equiv c_- \,$ on $\,\partial \digamma_-\,$. In contrast to Theorem  \ref{pharmonicGradient},   gradient of $\,u\,$ may vanish at some points in $\,\Delta\,$, even when $\,c_+ = c_- \bydef c \not = c_0\,$. In other words, Theorem \ref{pharmonicGradient} fails already for $\,p = 2\,$ when $\,\digamma\,$ is not connected. Our example is as follows.

\begin{example}\label{ThreeConductors} Consider a harmonic function
\begin{equation} \label{aHarmonicFunction}\,u(z) = \log |z^2 - 1| \;,\,\, \textnormal{defined for} \;\; z \not = \pm 1\, .
\end{equation}
For a given parameter $\, \lambda \in \mathbb R\,$ the  level set $\, \mathscr L_\lambda \,\bydef \{ z ;\; u(z)  = \lambda \}\, $ is a quartic curve defined as the locus of points whose product of the distances to two points $\,\pm 1\,$ (referred to as foci) is constant equal to $\, e^\lambda\,$.  This is none other than the familiar \textit{Cassini oval},  named after the astronomer Giovanni Domenico Cassini \cite{Cassini} .  The interested reader is referred to the catalog of Dennis Lawrence  \cite{Lawrence}  and \cite {Gray} for details. 
Cassini oval is a   $\,\mathscr C^\infty\,$-smooth Jordan curve when $\,\lambda > 0\,$. It becomes a lemniscate of Bernoulli  when $\,\lambda = 0\,$,  and consists of two  $\,\mathscr C^\infty\,$-smooth Jordan curves when $\,\lambda < 0\,$. \\

Choose and fix  parameters $\, 0 < r < 1 < R\,$.  Our simply connected domain $\,\Omega \,$  is enclosed by the $\,\mathscr C^\infty\,$-smooth Jordan curve 
\begin{equation}
\partial \Omega  \,\bydef\, \{ z ; \; |\,z^2  - 1 \,| \, = R\,\} \;\;,\; \;\;u \equiv  \log R\, > 0\;\; \textnormal{ on} \; \partial \Omega
\end{equation}
Whereas the boundaries of $\,\digamma_+ \,$ and $\,\digamma_-  \,$ are determined by the equation 

\begin{equation}
\partial \digamma_\pm  \,=\, \{ z ; \; |\,z^2  - 1 \,| \, = r\,\}\,,  \;\;,\; \;\;u \equiv  \log r\, < 0\;\; \textnormal{ on} \; \partial \digamma_\pm
\end{equation}
whenever  $\,\Re z \,$ is positive or negative, respectively. 
The critical level curve is the leminiscate of Bernoulli passing through the origin. 
See Figure \ref{Cassini Capacitor}

\end{example}

\begin{center}
\begin{figure}[h] 
 \includegraphics[angle=0, width=0.90 \textwidth]{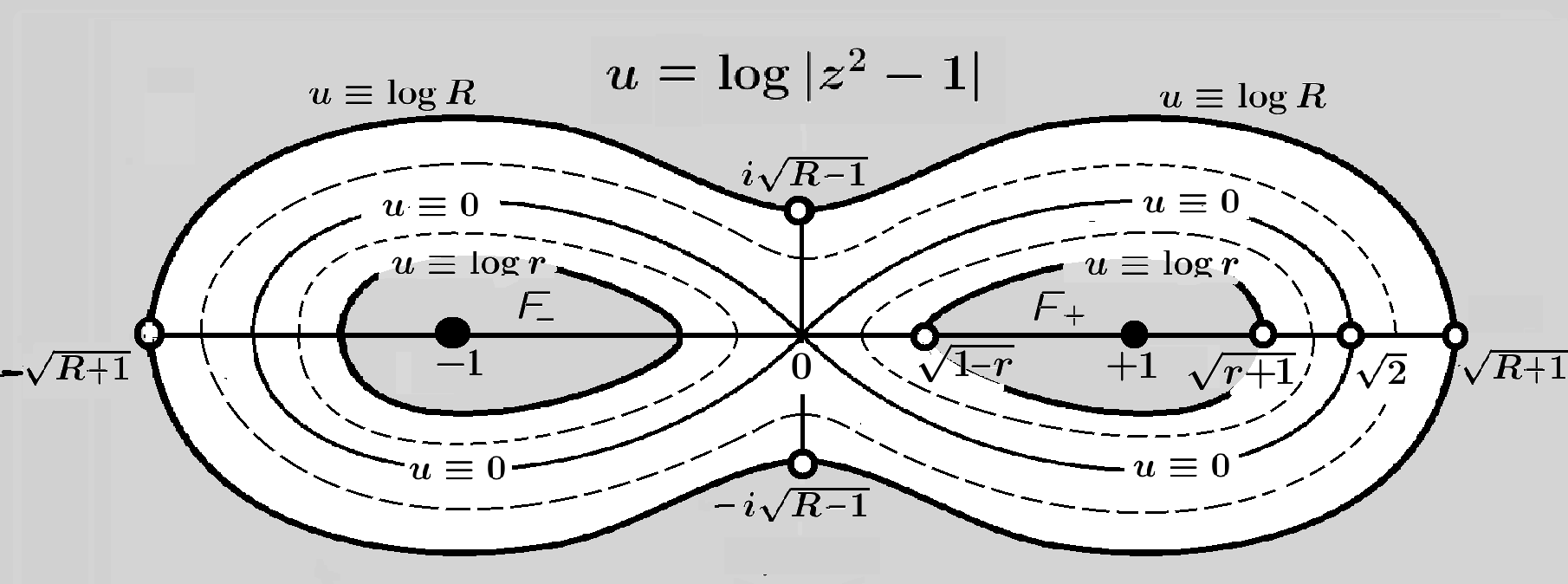}
\caption{ Cassini Capacitor and the critical level curve $\,u \equiv 0\,$ along a lemniscate of Bernoulli of foci $\,\pm 1\,$ } \label{Cassini Capacitor}
\end{figure}
\end{center}
\newpage

\bibliographystyle{amsplain}

\begin{thebibliography}{9}













\bibitem{Alessandrini}
G. Alessandrini,  {\em Critical points of solutions of elliptic equations in two variables}, Ann. Scuala Norm. Suo. Pisa C1. Sci. (4) 14 (1987), no. 2, 229--256 (1988)




\bibitem{AS}
G. Alessandrini and M. Sigalotti, {\em Geometric properties of solutions to the anisotropic $p$-Laplace equation in dimension two},
Ann. Acad. Sci. Fenn. Math. {\bf 26} (2001), no. 1, 249--266.\\





\bibitem{Antman}
S.S. Antman {\em Fundamental mathematical problems in the theory of nonlinear elasticity}. Univ. Maryland,  College Park, Md., (1975),  35--54.




\bibitem{AL}
G. Aronsson and P. Lindqvist, {\em On $\,p\,$-Harmonic Functions in the Plane and Their Stream Functions}, Journal of Differential Equations {\bf 74} (1988), 157--178.












\bibitem{AIMb}
K. Astala, T. Iwaniec, and G. Martin, {\em Elliptic partial differential equations and quasiconformal mappings in the plane},  Princeton University Press, Princeton, NJ, 2009.


\bibitem{Ball1}
J. Ball, {\em Convexity conditions and existence theorems in nonlinear elasticity},  Arch. Rational Mech. Anal. 63 (1976/77), no. 4,  337--403.

\bibitem{Ball2}
J. Ball, {\em Constitutive inequalities and existence theorems in nonlinear elastostatics},  Nonlinear analysis and mechanics: Heriot--Watt Symposium (Edinburgh, 1976), Vol. I, pp. 187--241. 
Res. Notes in Math., No. 17, Pitman, London, (1977).

\bibitem{Ball3}
J. Ball, {\em Global invertibility of Sobolev functions and the interpenetration of matter}, Proc. Roy.Soc. Edinburgh, Sect.A 88 (1981), 315--328.

\bibitem{Ball4}
J. Ball, {\em Some open problems in elasticity}, Geometry, mechanics, and dynamics,  3--59, Springer, New York, 2002.

\bibitem{BojIwanPreprint}
B. Bojarski and T. Iwaniec, {\em
$p$-harmonic equation and quasiregular mappings},  preprint (SFB 72) 617 (1983), Universitat Bonn.


\bibitem{BI}
B. Bojarski and T. Iwaniec, {\em
$p$-harmonic equation and quasiregular mappings. Partial differential equations},
Banach Center Publ., 19, PWN, Warsaw, (1987) 25--38.

\bibitem{Cassini}
J.D. Cassini, {\em De l’Origine et du progrès de l’astronomie et de son usage dans la géographie et dans la navigation}. L’Imprimerie Royale. (1693). pp. 36.

\bibitem{Choquet}
G. Choquet, {\em Sur un type de transformation analytique
g\'{e}n\'{e}ralisant la repr\'{e}sentation conforme et d\'{e}finie
au moyen de fonctions harmoniques,} Bull. Sci. Math., {\bf 69},
(1945), 156-165.


\bibitem{Ciarlet}
P. G. Ciarlet, {\em Mathematical elasticity Vol. 1.  Three--dimensional elasticity},  Studies in Mathematics and its Applications, 20,North--Holland Publishing Co., Amsterdam, 1988.





\bibitem{DurenBook}
P. Duren, \textit{Harmonic mappings in the plane},  Cambridge University Press, Cambridge, (2004).



\bibitem{DurenHengartner}
P. Duren and W. Hengartner, {\em Harmonic mappings of multiply connected domains}, Pacific J. Math. {\bf 180} (1997), no. 2, 201--220.

\bibitem{GleasonWolff}
S. Gleason, and T. Wolff {\em Levy's harmonic gradient maps in higher dimensions}, Communications in Partial Differential Equations, 16 (12), 1925--1968 (1991).


\bibitem{Gray} A. Gray, E. Abbena, S. Salamon  "Cassinian Ovals." §4.2 in Modern Differential Geometry of Curves and Surfaces with Mathematica, 2nd ed. Boca Raton, FL: CRC Press, pp. 82-86, 1997.

\bibitem{GM} A. G. Gurevich, G. A. Melkov \textit{Magnetization Oscillations and Waves} : CRC Press, 1996.



\bibitem{HKMb}
J. Heinonen, T. Kilpel\"ainen, and O. Martio, \textit{Nonlinear potential theory of degenerate elliptic equations}, Oxford University Press, New York, 1993.




\bibitem{IKO}
T.  Iwaniec, A.  Koski, and J.  Onninen, {\em Isotropic p-harmonic systems in 2D Jacobian estimates and univalent solutions}, Rev. Mat. Iberoam. {\bf 32} (2016), no. 1, 57--77.

\bibitem{IKOdiff}
 T. Iwaniec, L. V.  Kovalev, and J. Onninen, \textit{Diffeomorphic approximation of Sobolev homeomorphisms},  Arch. Ration. Mech. Anal. {\bf 201} (2011), no. 3, 1047--1067.

\bibitem{NitscheConjecture}
 T. Iwaniec, L. V.  Kovalev, and J. Onninen, \textit{The Nitsche Conjecture},  Journal of the American Mathematical Society
Vol. 24, No. 2 (2011), pp. 345-373.


\bibitem{IM}
T. Iwaniec and J. J. Manfredi, {\em  Regularity of p-harmonic functions on the plane},  Rev. Mat. Iberoamericana {\bf 5} (1989), no. 1-2, 1--19.


\bibitem{IOmono}
 T. Iwaniec  and J. Onninen, \textit{Monotone Sobolev mappings of planar domains and surfaces} Arch. Ration. Mech. Anal. {\bf 219} (2016), no. 1, 159--181.




\bibitem{IOrkc}
 T. Iwaniec  and J. Onninen,  \textit{Rad\'o-Kneser-Choquet theorem},  Bull. Lond. Math. Soc. {\bf 46} (2014), no. 6, 1283--1291.



\bibitem{IOsimplyconnectedRKC}
 T. Iwaniec  and J. Onninen,  \textit{Rad\'o-Kneser-Choquet theorem for simply connected domains ($\,p\,$-harmonic setting},  Transactions of AMS, vol. 371, no. 4 (2019), 2307--2341.



\bibitem{IOLimitsSobolevHom}
 T. Iwaniec  and J. Onninen,  \textit{Limits of Sobolev homeomorphisms},  J.Eur.M.S., vol.19, no.2, (2017),  473--505.

\bibitem{JB} 
M. Jayalakshmi and K. Balasubramanian \textit{Simple Capacitors to Supercapacitors- An Overwiew}, International Journal of Electrochimical Science, vol. 3, (2008), 1196--1217.

\bibitem{Ja}
S. Janiszewski, {\em Sur les coupures du plan faites par des continus}, Prace Mat.-Fiz., {\bf 26} (1913), 11--63.


\bibitem{Khitun}
A. Khitun, \textit{An entertaining physics: On the possibility of energy storage enhancement in electrostatic capacitors using the compensational inductive electric field}, Applied Physics Letters 117 (15), 2020.


\bibitem{Kn}
H. Kneser, {\em L\"{o}sung der Aufgabe 41,} Jahresber. Deutsch.
Math.-Verein., {\bf 35}, (1926), 123--124.

\bibitem{Kub}
K. Kuratowski,  {\em Topology. Vol. II},  Academic Press, New York-London; Pa\'nstwowe Wydawnictwo Naukowe Polish Scientific Publishers, Warsaw (1968).







\bibitem{Laugesen}
 R.S. Laugesen {\em Injectivity can fail for higher--dimensional harmonic extensions},  Complex Variables Theory Appl. 28 (1996), no. 4, 357--369.


\bibitem{Lawrence}
J. D.  Lawrence. {\em A catalog of special plane curves}, Dover Publications (1972). pp. 5, 153–155. ISBN 0-486-60288-5.

\bibitem{Lewy}
H. Lewy, {\em On the non-vanishing of the jacobian of a homeomorphism by harmonic gradients}, Ann. of Math. (2) {\bf 88} 1968 518--529. 



\bibitem{Lyzzaik}
A. Lyzzaik,  {\em Univalence criteria for harmonic mappings in multiply connected domains}, J. London Math. Soc. (2) {\bf 58} (1998), no. 1, 163--171.


\bibitem{Mc}
L.F. McAuley,  {\em Some fundamental theorems and problems related to monotone mappings} In: Proceedings of First Conference on Monotone Mappings and Open Mappings. State Univ. of New York at Binghamton, N.Y., (1971).

\bibitem{McAuley}
L.F. McAuley,  {\em Monotone mappings-some milestones}, General topology and modern analysis, Academic Press, New York - London , 1981, pp. 117--141.


 \bibitem{Morrey}
C. B. Morrey,  {\em The Topology of (Path) Surfaces}, Amer. J. Math. {\bf 57} (1935), no. 1, 17--50.

 \bibitem{Nitsche}
J.C.C. Nitsche  {\em On the modulus of doubly connected regions under harmonic mappings}, Amer.  Math. Monthly, 69, (1962), pp. 781--782.


\bibitem{PAT} K. Poonam,  Sharma, A. Arora, S. K. Tripathi \textit{Review of supercapacitors: Materials and devices}, Journal of Energy Storage, vol. 21, 2019, pp. 801--825.
 
\bibitem{Rad}
T. Rad\'{o}, \textit{Aufgabe 41.}, Jahresber. Deutsch. Math.-Verein.,
{\bf 35}, (1926), 49.


\bibitem{Whyburn}
G.T. Whyburn, \textit{Monotonicity of limit mappings}, Duke Math. J. {\bf 29}, (1962), 465--470.







\end{thebibliography}

\end{document}